\documentclass[reqno]{amsart}
\usepackage{amsthm,amsmath,amssymb}
\usepackage{stmaryrd,mathrsfs}
\usepackage[T1]{fontenc}
\usepackage{esint}

\allowdisplaybreaks

\newcommand{\ud}[0]{\,\mathrm{d}}

\newcommand{\abs}[1]{|#1|}
\newcommand{\Babs}[1]{\Big|#1\Big|}
\newcommand{\norm}[2]{|#1|_{#2}}

\newcommand{\Norm}[2]{\|#1\|_{#2}}

\newcommand{\BNorm}[2]{\Big\|#1\Big\|_{#2}}
\newcommand{\pair}[2]{\langle #1,#2 \rangle}

\newcommand{\ave}[1]{\langle #1\rangle}

\newcommand{\bddlin}[0]{\mathscr{L}}

\newcommand{\BMO}[0]{\operatorname{BMO}}
\newcommand{\supp}[0]{\operatorname{supp}}

\newcommand{\sign}[0]{\operatorname{sgn}}

\newcommand{\R}{\mathbb{R}}

\newcommand{\N}{\mathbb{N}}
\newcommand{\Z}{\mathbb{Z}}
\newcommand{\T}{\mathbb{T}}

\def\cyr{\fontencoding{OT2}\fontfamily{wncyr}\selectfont}
\DeclareTextFontCommand{\textcyr}{\cyr}

\swapnumbers 

\numberwithin{equation}{section}

\theoremstyle{plain}
\newtheorem{theorem}[equation]{Theorem}
\newtheorem{proposition}[equation]{Proposition}

\newtheorem{lemma}[equation]{Lemma}

\theoremstyle{definition}

\theoremstyle{remark}
\newtheorem{remark}[equation]{Remark}

\makeatletter
\@namedef{subjclassname@2010}{%
  \textup{2010} Mathematics Subject Classification}
\makeatother

\usepackage{hyperref} 
\hypersetup{
    colorlinks=true,       
    linkcolor=blue,          
    citecolor=magenta,        
    filecolor=magenta,      
    urlcolor=cyan           
}

\begin{document}

\title[Convergence of Walsh--Fourier series]{Pointwise convergence of Walsh--Fourier series of vector-valued functions}

\author[T. P. Hyt\"onen]{Tuomas P. Hyt\"onen}
\address{Department of Mathematics and Statistics, P.O.B.~68 (Gustaf H\"all\-str\"omin katu~2b), FI-00014 University of Helsinki, Finland}
\email{tuomas.hytonen@helsinki.fi}
\thanks{T.H. is supported by the European Union through the ERC Starting Grant ``Analytic-probabilistic methods for borderline singular integrals'', and by the Academy of Finland, projects 130166 and 133264.}

\author[M. T. Lacey]{Michael T. Lacey}
\address{School of Mathematics \\
Georgia Institute of Technology \\
Atlanta GA 30332 }
\email{lacey@math.gatech.edu}
\thanks{M.L. supported in part by the NSF grant 0968499, and a grant from the Simons Foundation (\#229596 to Michael Lacey). }

\date{\today}

\subjclass[2010]{42B20, 46E40}


\begin{abstract}{
We prove a version of Carleson's Theorem in the Walsh model for vector-valued functions:
For $1<p< \infty$, and a UMD space $Y$, the Walsh-Fourier series of $f \in L ^{p}(0,1;Y)$ converges pointwise, provided that  $Y$ is a complex interpolation space $Y=[X,H]_\theta$ between another UMD space $X$ and a Hilbert space $H$, for some $\theta\in(0,1)$. Apparently, all known examples of UMD spaces satisfy this condition.  }
\end{abstract}

\maketitle

\section{Introduction}

We are interested in the vector-valued extension of Carleson's celebrated theorem on pointwise convergence of Fourier series \cite{Carleson:66}, or more precisely, in this paper, on the variant due to Billard \cite{Billard} about Walsh--Fourier series. By `vector-valued' we understand functions that take their values in a possibly infinite-dimensional Banach space $X$. It is well known that the most general setting in which such results could be hoped for is when $X$ is a UMD (unconditionality of martingale differences) space.

So far, vector-valued pointwise convergence results of this nature only exist in the more restricted class of UMD spaces with an unconditional basis, or somewhat more generally, in UMD lattices. Indeed, Carleson's theorem in such spaces was proven by Rubio de Francia \cite{RdF:Studia,RdF:LNM}, and Billard's theorem by Weisz \cite{Weisz:Vilenkin}, who also treated the more general Vilenkin--Fourier series. (The abstract and the MR review of the last-mentioned paper misleadingly claim the result for UMD spaces, although it is only proven assuming an unconditional basis.) All these results ultimately rely on the classical Carleson (or Billard) theorem as a black box: the scalar-valued boundedness of the relevant maximal partial sum operator $S^*$ is applied component-wise in the unconditional basis (or pointwise in a representation of the lattice as a function space).


Rubio de Francia explicitly raised the following question \cite[Problem 4 on p.~220]{RdF:LNM}:
\begin{quote}
  It would be interesting to know if $B$-valued Fourier series converge a.e.\ for $B\in UMD$ ($B$ not a lattice), e.g., for the Schatten ideals: $B=C_p$, $1<p<\infty$.
\end{quote}
Apparently, no published progress on this was made in the last 25 years until the recent proof of the `little Carleson theorem' in general UMD spaces by Parcet, Soria and Xu \cite{PSX}: the sequence of partial sum $S_n f(x)$ of the Fourier series of $f\in L(\log L)^{1+\delta}(\T;X)$ grows at most at the rate $o(\log\log n)$ for a.e.~$x\in\T$. They adapt Carleson's original argument \cite{Carleson:66}, rather than just his result, to this vector-valued question.

In this paper, we obtain the first partial answer to the actual convergence issue. We prove the pointwise convergence of $Y$-valued Walsh--Fourier series for all UMD spaces $Y$ of the following special form: $Y$ is a complex interpolation space $Y=[X,H]_\theta$ between another UMD space $X$ and a Hilbert space $H$, where $\theta\in(0,1)$. This includes all UMD lattices \cite[Corollary on p.~216]{RdF:LNM}. It also includes the Schatten ideals $C_p$, $p\in(1,\infty)$, specifically raised in Rubio de Francia's question (for we can always pick another $q\in(1,\infty)$ so that $C_p=[C_q,C_2]_\theta$), and apparently all other known examples of UMD spaces as well. In fact, Rubio de Francia also asked \cite[Problem 4 on p.~220]{RdF:LNM}:
\begin{quote}
  Is every $B\in UMD$ intermediate between a ``worse'' $B_0\in UMD$ and a Hilbert space?
\end{quote}
This question also remains open. A possible affirmative answer, in combination with our present contribution, would yield the pointwise convergence of $X$-valued Walsh--Fourier series for every UMD space $X$. Conversely, a counterexample to the pointwise convergence result would be a counterexample to the mentioned interpolation property.

Rubio de Francia's class of intermediate UMD spaces $Y=[H,X]_\theta$ has played a role in a number of earlier works. Rubio de Francia himself indicated how the boundedness of linear operators with a decomposition
\begin{equation}\label{eq:TsumTj}
  T=\sum_{j\in\Z}T_j,\qquad \Norm{T_j}{\bddlin(L^2(\R;H))}\leq C2^{-\varepsilon\abs{j}},\qquad
  \Norm{T_j}{\bddlin(L^q(\R;X))}\leq C.
\end{equation}
can be conveniently handled in such spaces \cite[p.~219--220]{RdF:LNM}: one only needs the decay estimate in a Hilbert space, and a much cruder uniform estimate in general UMD spaces to conclude the summable decay $\Norm{T_j}{\bddlin(L^p(\R;Y))}\leq C2^{-\varepsilon'\abs{j}}$ by interpolation. The same class reappeared in Berkson--Gillespie \cite{BG:05} and Hyt\"onen \cite{Hyt:LPS}, where stronger results were obtained for such spaces than for general UMD spaces. See \cite{BG:05,Hyt:LPS} for more information on these spaces.

Although treated in the same paper, Rubio de Francia's extension of Carleson's theorem was not based on this interpolation property but on the explicit lattice structure in a more fundamental way. In contrast, our present contribution can be vaguely thought of as an adaptation of Rubio de Francia's approach on the operators \eqref{eq:TsumTj} to the maximal partial sum operator $S^*$ of the Walsh--Fourier series. The decomposition of $S^*$ is furnished by the time-frequency analysis of Lacey--Thiele \cite{LT:MRL}, and the estimates forming the basis of interpolation have a more subtle structure than above.

In fact, our proof is built in such a way that we obtain the convergence of Walsh--Fourier series for all UMD spaces $X$ satisfying a new condition, which we call the \emph{tile-type}, and we verify this condition for all intermediate UMD spaces as described. The name tile-type refers, on the one hand, to its resemblance of some established Banach space properties like type and martingale-type, and on the other hand, to its connection to the time-frequency tiles in the phase plane, as in the work of Lacey--Thiele \cite{LT:MRL}. The tile-type inequality is applied exactly once in the proof; everything else works for general UMD spaces. In this way, we single out for further investigation a specific sufficient condition for the convergence of vector-valued Walsh--Fourier series in full generality. 

The setting of a UMD space requires, ultimately, the use of martingale differences.
These are actually readily apparent in the Walsh case. 
The main point of departure from the classical reasoning is the notion of tile-type, and its use in the Size Lemma.
The remaining lemmas are known, but the details are included.

The extension of the present results to the trigonometric Fourier series will be treated in a subsequent work.

\section{Main results and preliminaries}

We introduce the Rademacher functions
\begin{equation*}
  r_i(x):=\sign\sin(2\pi\cdot 2^i x)=\sum_{k\in\N}\big(1_{2^{-i}[k,k+\tfrac12)}(x)-1_{2^{-i}[k+\tfrac12,k+1)}(x)\big)
\end{equation*}
and the Walsh functions
\begin{equation*}
  w_n(x):=\prod_{i=0}^{\infty}r_i(x)^{n_i},\qquad\text{for}\qquad n=\sum_{i=0}^{\infty}n_i 2^i\in\N,\quad n_i\in\{0,1\},
\end{equation*}
as objects defined for all $x\in\R_+$. The restrictions $1_{[0,1)}w_n$ form an orthonormal basis of $L^2(0,1)$.

Our main result is the following:

\begin{theorem}\label{thm:main}
	Let $Y$ be an intermediate UMD space, $p\in(1,\infty)$, and $f\in L^p(0,1;Y)$. Then
\begin{equation*}
   S_N f(x):=\sum_{n=0}^{N-1} \pair{f}{w_n}w_n(x)\to f(x)
\end{equation*}
as $N\to\infty$ for a.e.~$x\in(0,1)$. In fact,  the  maximal partial sum operator $S^*$,
\begin{equation*}
  S^*f(x):=\sup_{N\in\N}\abs{S_N f(x)},
\end{equation*}
is bounded from $L^p(0,1;Y)$ to $L^p(0,1)$.
\end{theorem}

Making $N$ a function $N(x)$, we arrive at the linearization $S_{N(x)}f(x)$, and the above theorem is equivalent to the uniform bound
\begin{equation*}
  \Norm{S_{N(\cdot)}f}{L^p(0,1;Y)}\leq C\Norm{f}{L^p(0,1;Y)}
\end{equation*}
for all $f$ and $N$. To express $S_{N(x)}f(x)$ in a more flexible form, we recall more notation.

A \emph{tile} is a dyadic rectangle $P\subset\R_+\times\R_+$ of area $1$, i.e.,
\begin{equation*}
  P=I\times\omega=I\times\frac{1}{\abs{I}}[n,n+1),\qquad I\in\mathscr{D},\quad n\in\N,
\end{equation*}
where $\mathscr{D}$ is the collection of dyadic intervals of $\R_+$.
To every tile $P$, we associate the wave packet
\begin{equation*}
  w_P(x):=\frac{1}{\abs{I}^{1/2}}w_P^\infty(x),\qquad w_P^\infty(x)=1_I(x)w_n\Big(\frac{x}{\abs{I}}\Big).
\end{equation*}
The superscript $\infty$ refers to $L^\infty$ normalization. The Haar functions arise as special cases:
\begin{equation*}
  h_I(x)=\frac{1}{\abs{I}^{1/2}}1_I(x)r_0\Big(\frac{x}{\abs{I}}\Big)=w_{I\times\abs{I}^{-1}[0,1)}(x).
\end{equation*}

A \emph{bitile} is a dyadic rectangle of area $2$, i.e.,
\begin{equation*}
\begin{split}
  P &=I\times\frac{1}{\abs{I}}[2n,2(n+1)) \\
  &=I\times\frac{1}{\abs{I}}[2n,2n+1)\cup I\times\frac{1}{\abs{I}}[2n+1,2(n+1))
  =:P_d\cup P_u,
\end{split}
\end{equation*}
where the second line gives the canonical decomposition of $P$ to its \emph{down-tile} and \emph{up-tile}. If $P=I\times\omega$ is either a tile or a bitile, we write $I_P:=I$ and $\omega_P:=\omega$ for its time and frequency interval, respectively.

The following identity is explained in Thiele \cite[p.~68--69]{Thiele:book}:
\begin{equation*}
  S_{N(x)}f(x)=\sum_{\substack{P\text{ bitile}\\ I_P\subseteq[0,1)}}\pair{f}{w_{P_d}}w_{P_d}(x) 1_{\omega_{P_u}}(N(x)).
\end{equation*}
As in \cite{Thiele:book}, we will drop the restriction that $I_P\subseteq[0,1)$ in the subsequent analysis, and consider the resulting scale-invariant operator on $L^p(\R_+;Y)$ rather than $L^p(0,1;Y)$. We will first establish the following inequality 
on the bilinear form 
\begin{equation*}
	\bigl\lvert  \langle S_N f, g1 _E  \rangle\bigr\rvert \lesssim \lVert f\rVert_{L^q (\mathbb R _+; Y)} \Norm{g}{L^{\infty}(\mathbb{R}_+;Y^*)} \lvert  E\rvert ^{1/q} \,, 
\end{equation*}
where $ q$ is the tile-type of the UMD space $Y$.  This proof is then refined to prove the full range of estimates 
for the Carleson operator.

A partial order (among either tiles or bitiles) is defined by
\begin{equation*}
\begin{split}
  P\leq P'\ &\overset{\operatorname{def}}{\Leftrightarrow}\ I_P\subseteq I_{P'}\ \text{and}\ \omega_{P}\supseteq\omega_{P'} \\
  &\Leftrightarrow\ P_d\leq P_d'\ \text{or}\ 
     P_u\leq P_u'.
\end{split}
\end{equation*}
For bitiles, we also define
\begin{equation*}
   P\leq_j P'\ \overset{\operatorname{def}}{\Leftrightarrow}\ P_j\leq P_j',\qquad j\in\{d,u\}.
\end{equation*}

A tree $\mathbf{T}$ is a collection of bitiles $P$ for which there exists a top bitile $T$ (not necessarily an element of $\mathbf{T}$) such that
\begin{equation*}
  P\leq T\qquad\forall P\in\mathbf{T}.
\end{equation*}
Down-trees and up-trees are defined similarly by replacing $\leq$ by $\leq_d$ or $\leq_u$.

\begin{lemma}\label{lem:treeIdentity}
Let $\mathbf{T}$ be an up-tree with top $T$. Then for all $P\in\mathbf{T}$, we have
\begin{equation*}
  w_{P_d}(x) = 
  \epsilon_{PT}\cdot  w_{T_u}^\infty(x)\cdot   h_{I_P}(x)
\end{equation*}
for some constant factor $\epsilon_{PT}\in\{-1,+1\}$. Hence in particular
\begin{equation*}
  \pair{f}{w_{P_d}}w_{P_d} = 
\pair{f\cdot w_{T_u}^\infty}{h_{I_P}}h_{I_P}\cdot w_{T_u}^\infty.
\end{equation*}
\end{lemma}

\begin{proof}
We have $T_u=I_T\times\abs{I_T}^{-1}[n_T,n_T+1)$, with odd $n_T$. Consider an element $P\in\mathbf{T}$ with $P_u=I_P\times\abs{I_P}^{-1}[n_P,n_P+1)$, again with odd $n_P$, and let $2^{-k}:=\abs{I_P}/\abs{I_T}$. Then $P_u\leq T_u$ says that
\begin{equation*}
  I_P\subseteq I_T\qquad\text{and}\qquad
  2^{-k}(n_T+1)-1\leq n_P\leq 2^{-k}n_T.
\end{equation*}
If $n_T=\sum_{i=0}^{\infty}2^i n_i$, then the unique integer value of $n_P$ in the given range is
\begin{equation*}
  n_P=\sum_{i=k}^{\infty} 2^{i-k}n_i,
\end{equation*}
which is odd if and only if $n_k=1$. For those values of $k$, we have $P_d=I_P\times\abs{I_P}^{-1}[n_P-1,n_P)$, where
\begin{equation*}
  n_P-1=\sum_{i=k+1}^{\infty} 2^{i-k}n_i.
\end{equation*}
Hence
\begin{equation*}
\begin{split}
  w_{P_d}
  &=\frac{1_{I_P}}{\abs{I_P}^{1/2}}w_{n_P-1}\Big(\frac{\cdot}{\abs{I_P}}\Big)
  =\frac{1_{I_P}}{\abs{I_P}^{1/2}}w_{2^k (n_P-1)}\Big(\frac{\cdot}{\abs{I_T}}\Big) \\
  &=\frac{1_{I_P}}{\abs{I_P}^{1/2}}\prod_{i=k+1}^{\infty}r_i\Big(\frac{\cdot}{\abs{I_T}}\Big)^{n_i} \\
  &\overset{(*)}{=}\frac{1_{I_P}}{\abs{I_P}^{1/2}}\prod_{i=0}^{\infty}r_i\Big(\frac{\cdot}{\abs{I_T}}\Big)^{n_i}\times r_k\Big(\frac{\cdot}{\abs{I_T}}\Big)\times
        \prod_{i=0}^{k-1}r_i\Big(\frac{\cdot}{\abs{I_T}}\Big)^{n_i} \\  
        &=1_{I_T}w_{n_T}\Big(\frac{\cdot}{\abs{I_T}}\Big)\times\frac{1_{I_P}}{\abs{I_P}^{1/2}} r_{0}\Big(\frac{\cdot}{\abs{I_P}}\Big)\times
        \prod_{i=0}^{k-1}r_i\Big(\frac{\cdot}{2^k\abs{I_P}}\Big)^{n_i} \\
  &=w_{T_u}^\infty\times h_{I_P}\times
        \prod_{i=0}^{k-1}r_i\Big(\frac{\cdot}{2^k\abs{I_P}}\Big)^{n_i} .    
\end{split}
\end{equation*}
Note that $n_k=1$ was used in $(*)$, together with $r_i^2\equiv 1$. Notice that the last product takes a constant value on $I_P$, as $r_i$ is constant over dyadic intervals of length $2^{-i-1}$; this is our $\epsilon_{PT}$. The second claim follows from $\epsilon_{PT}^2=1$.
\end{proof}

\section{The tile-type of a Banach space}

Let $\mathscr{T}$ be a collection of up-trees such that: For any two distinct pairs $(P^i,\mathbf{T}^i)$ with $P^i\in\mathbf{T}^i\in\mathscr{T}$, we have $P^1_d\cap P^2_d=\varnothing$. We say that a Banach space $X$ has \emph{tile-type} $q$ if the following estimate holds uniformly for all such $\mathscr{T}$ and all $f\in L^q(\R_+;X)$:
\begin{equation*}
  \Big(\sum_{\mathbf{T}\in\mathscr{T}}\BNorm{\sum_{P\in\mathbf{T}}\pair{f}{w_{P_d}}w_{P_d}}{L^q(\R_+;X)}^q\Big)^{1/q}
  \lesssim\Norm{f}{L^q(\R_+;X)}.
\end{equation*}

Our results about this concept are summarized in the following proposition. It shows in particular that tile-type behaves somewhat like the classical cotype.

\begin{proposition}  \label{p.type}
A necessary condition for tile-type $q$ is that $X$ is a UMD space and $q\geq 2$. If a UMD space has tile-type $q$, it has tile-type $p$ for all $p\in[q,\infty)$. Every Hilbert space has tile-type $2$, and every complex interpolation space $[X,H]_\theta$, $\theta\in(0,1)$, between a UMD space and a Hilbert space has tile-type $2/\theta$.
\end{proposition}

In particular, every $L^p$ space (even non-commutative) has tile-type $q$ for all $q\in(\max\{p,p'\},\infty)$.

We consider the following operators:
\begin{equation*}
  \mathscr{W}_{\mathscr{T}}f:=\Big\{\sum_{P\in\mathbf{T}}\pair{f}{w_{P_d}}w_{P_d}\Big\}_{\mathbf{T}\in\mathscr{T}},\qquad
  \mathscr{W}_{\mathscr{T}}'f:=\Big\{\sum_{P\in\mathbf{T}}\pair{f}{w_{P_d}}h_{I_P}\Big\}_{\mathbf{T}\in\mathscr{T}}.
\end{equation*}
We are concerned about the boundedness
\begin{equation*}
  \mathscr{W}_{\mathscr{T}}: L^p(\R_+;X)\to \ell^p(\mathscr{T};L^p(\R_+;X)).
\end{equation*}
From Lemma~\ref{lem:treeIdentity} it follows that
\begin{equation*}
  \Norm{\mathscr{W}_{\mathscr{T}}f}{ \ell^p(\mathscr{T};L^p(\R_+;X))}
  =\Norm{\mathscr{W}_{\mathscr{T}}'f}{ \ell^p(\mathscr{T};L^p(\R_+;X))},
\end{equation*}
so the question is equivalent for $\mathscr{W}_{\mathscr{T}}$ and $\mathscr{W}_{\mathscr{T}}'$. However, the latter operator will be more amenable for the end-point mapping property
\begin{equation*}
  \mathscr{W}_{\mathscr{T}}': L^\infty(\R_+;X)\to \ell^\infty(\mathscr{T};\BMO(\R_+;X)),
\end{equation*}
which will play a role in interpolation. Note that $\BMO$ stands for the dyadic BMO, since this is the only BMO space we need here.

\begin{lemma}\label{lem:WonH}
If $H$ is a Hilbert space, then
\begin{equation*}
  \Norm{\mathscr{W}_{\mathscr{T}}'f}{\ell^2(\mathscr{T};L^2(\R_+;H))}=\Norm{\mathscr{W}_{\mathscr{T}}f}{\ell^2(\mathscr{T};L^2(\R_+;H))}
  \leq \Norm{f}{L^2(\R_+;H)}.
\end{equation*}
\end{lemma}

\begin{proof}
This follows from the fact that all appearing $w_{P_d}$ are pairwise orthogonal, and hence
\begin{equation*}
\begin{split}
  \Big(\sum_{\mathbf{T}\in\mathscr{T}}\BNorm{\sum_{P\in\mathbf{T}}\pair{f}{w_{P_d}}w_{P_d}}{L^2(\R_+;H)}^2\Big)^{1/2}
  &=\Big(\sum_{\mathbf{T}\in\mathscr{T}}\sum_{P\in\mathbf{T}}\norm{\pair{f}{w_{P_d}}}{H}^2\Big)^{1/2} \\
  &\leq\Norm{f}{L^2(\R_+;H)}.\qedhere
\end{split}
\end{equation*}
\end{proof}

\begin{lemma}\label{lem:WonBMO}
If $X$ is a UMD space, then
\begin{equation*}
  \Norm{\mathscr{W}_{\mathscr{T}}'f}{\ell^\infty(\mathscr{T};\BMO(\R_+;X))}
  \lesssim \Norm{f}{L^\infty(\R_+;X)}.
\end{equation*}
\end{lemma}

\begin{proof}
It suffices to consider a single up-tree $\mathbf{T}$. By Lemma~\ref{lem:treeIdentity},
\begin{equation*}
  \sum_{P\in\mathbf{T}}\pair{f}{w_{P_d}}h_{I_P}
  =\sum_{P\in\mathbf{T}}\epsilon_{PT}\pair{f\cdot w_{T_u}^\infty}{h_{I_P}}h_{I_P}
\end{equation*}
is a martingale transform of $f\cdot w_{T_u}$. It is well known that martingale transforms map $L^\infty(\R_+;X)$ to $\BMO(\R_+;X)$ when $X$ is a UMD space. Since $\Norm{f\cdot w_{T_u}^\infty}{\infty}=\Norm{f}{\infty}$, the result follows.
\end{proof}

\begin{remark}
A similar argument shows that
\begin{equation*}
  \Norm{\mathscr{W}_{\mathscr{T}}f}{\ell^\infty(\mathscr{T};L^p(\R_+;X))}=\Norm{\mathscr{W}_{\mathscr{T}}'f}{\ell^\infty(\mathscr{T};L^p(\R_+;X))}
  \lesssim \Norm{f}{L^p(\R_+;X)}
\end{equation*}
for any $p\in(1,\infty)$ and any UMD space $X$. However, we have no use for this result, where the exponents of $\ell^\infty$ and $L^p$ do not match.
\end{remark}

\begin{lemma}
If $Y=[X,H]_{\theta}$ is a complex interpolation space between a UMD space $X$ and a Hilbert space $H$, with parameter $\theta\in(0,1)$, then
\begin{equation*}
  \Norm{\mathscr{W}_{\mathscr{T}}f}{\ell^p(\mathscr{T};L^p(\R_+;Y))}=\Norm{\mathscr{W}_{\mathscr{T}}'f}{\ell^p(\mathscr{T};L^p(\R_+;Y))}
  \lesssim \Norm{f}{L^p(\R_+;Y)}
\end{equation*}
holds for all $p\in[2/\theta,\infty)$.
\end{lemma}

\begin{proof}
Consider the operator $\mathscr{W}_{\mathscr{T}}'$, the result (but not the proof) for the other operator being equivalent. For $p=2/\theta$, we interpolate between the estimates of Lemmas~\ref{lem:WonH} and \ref{lem:WonBMO}, using the complex interpolation results
\begin{equation*}
  [L^\infty(\R_+;X),L^2(\R_+;H)]_\theta
  =L^p(\R_+;[X,H]_\theta)=L^p(\R_+;Y),
\end{equation*}
and
\begin{equation*}
\begin{split}
  [\ell^\infty(\mathscr{T}; &\BMO(\R_+;X)), \ell^2(\mathscr{T},L^2(\R_+;H))]_\theta \\
    &=\ell^p(\mathscr{T};[\BMO(\R_+;X),L^2(\R_+;H)]_\theta) \\
    &=\ell^p(\mathscr{T};L^p(\R_+;[X,H]_\theta))
    =\ell^p(\mathscr{T};L^p(\R_+;Y)).
\end{split}
\end{equation*}
For $p\in(2/\theta,\infty)$, we similarly interpolate between the result just established for $p=2/\theta$, and the result of Lemma~\ref{lem:WonBMO} specialized to $X=Y$.
\end{proof}

\section{The tree lemma}

We take $ E\subset \mathbb R _+$, and for a collection of tiles $ \mathbf P$, define two quantities below.
\begin{equation*}
	\operatorname{density}(\mathbf{P}):=\sup_{P\in\mathbf{P}}\sup_{P'\geq P}\frac{\abs{I_{P'}\cap E_{P'}}}{\abs{I_{P'}}},\qquad E_{P'}:=E\cap\{x \;:\; N (x)\in\omega_{P'}\}.
\end{equation*}
\begin{equation*}
   \operatorname{size}(\mathbf{P}):=\sup_{\mathbf{T}\subseteq\mathbf{P}\text{ up-tree}}
   \Big(\frac{1}{\abs{I_{\mathbf{T}}}}\int\Babs{\sum_{P\in\mathbf{T}}\pair{f}{w_{P_d}}w_{P_d}}^q\ud x\Big)^{1/q}.
\end{equation*}

The `Tree Lemma' is the estimate below. We detail the proof, indicating the use of the 
UMD property at a point below.

\begin{proposition}\label{prop:treeLemma}
For each tree $\mathbf{T}$, we have
\begin{equation*}
  \sum_{P\in\mathbf{T}}\abs{\pair{f}{w_{P_d}}\pair{w_{P_d}}{g 1_{E_{P_u}}}}
  \lesssim\operatorname{size}(\mathbf{T})\operatorname{density}(\mathbf{T})\abs{I_T},
\end{equation*}
where
\begin{equation*}
	E_{P_u}:=E\cap\{x \;:\; N (x)\in\omega_{P_u}\}.
\end{equation*}
\end{proposition}

Let $\mathscr{J}$ be the collection of maximal dyadic intervals $J\subseteq \bigcup_{P\in\mathbf{T}}I_P\subseteq I_T$ which do not contain any $I_P$, $P\in\mathbf{T}$. These intervals cover the set $\bigcup_{P\in\mathbf{T}}I_P$. Hence, for a choice of complex numbers $\abs{\epsilon_P}=1$,
\begin{equation}\label{eq:splitDomain}
\begin{split}
   \sum_{P\in\mathbf{T}}\abs{\pair{f}{w_{P_d}}\pair{w_{P_d}}{g 1_{E_{P_u}}}}
   &\leq\BNorm{\sum_{P\in\mathbf{T}}\epsilon_P\pair{f}{w_{P_d}}w_{P_d}1_{E_{P_u}}}{L^1(\R_+;X)} \\
   =\sum_{J\in\mathscr{J}}\BNorm{\cdots}{L^1(J;X)}
   &=\sum_{J\in\mathscr{J}}\BNorm{\sum_{\substack{P\in\mathbf{T} \\ I_P\supsetneq J}}\epsilon_P\pair{f}{w_{P_d}}w_{P_d}1_{E_{P_u}}}{L^1(J;X)}.
 \end{split}
\end{equation}

\begin{lemma}\label{lem:GJ}
For a fixed $J\in\mathscr{J}$, the subset
\begin{equation*}
  G_J:=J\cap\bigcup_{\substack{P\in\mathbf{T}\\ I_P\supsetneq J}} E_{P_u}
\end{equation*}
satisfies $\abs{G_J}\leq 2\operatorname{density}(\mathbf{T})\abs{J}$.
\end{lemma}

\begin{proof}
	Consider the dyadic parent $\hat{J}$ of $J$. By maximality of $J$, we have $\hat{J}\supseteq I_{\tilde P}$ for some $\tilde{P}\in\mathbf{T}$. Let $\hat{\omega}$ be the dyadic interval of size $2/\abs{\hat{J}}$ such that $\omega_{\tilde{P}}\supseteq\hat\omega\supseteq\omega_{T}$, where $T$ is the top of $\mathbf{T}$, so that the bitile  $\hat{P}:=\hat{J}\times\hat\omega$ satisfies $\tilde{P}\leq\hat{P}\leq T$. Now we claim that
\begin{equation}\label{eq:GJsubset}
  G_J\subseteq J\cap E_{\hat{P}}.
\end{equation}

Indeed, consider one of the $P$ appearing in $G_J$. Then $P\in\mathbf{T}$, thus $I_P\subseteq I_T$ and $\omega_P\supseteq\omega_T$, and also $I_P\supsetneq J$, thus $I_P\supseteq \hat{J}$. We also have $\abs{\omega_P}=2/\abs{I_P}\leq 2/\abs{\hat{J}}=\abs{\hat\omega}$, and $\omega_P\cap\hat\omega\supseteq\omega_T\neq\varnothing$, hence $\omega_P\subseteq\hat\omega$. But this means that
\begin{equation*}
  E_{P_u}=E\cap\{N\in\omega_{P_u}\}\subseteq E\cap\{N\in \omega_P\}\subseteq E\cap\{N\in\hat\omega\}=E_{\hat{P}},
\end{equation*}
which proves the claim~\eqref{eq:GJsubset}.

The proof is completed as follows, recalling that $\hat{P}\geq\tilde{P}\in\mathbf{T}$:
\begin{equation*}
\begin{split}
  \abs{G_J}\leq\abs{J\cap E_{\hat{P}}}
  &\leq\abs{\hat{J}}\frac{\abs{\hat{J}\cap E_{\hat{P}}}}{\abs{\hat{J}}}
  =2\abs{J}\frac{\abs{I_{\hat{P}}\cap E_{\hat{P}}}}{\abs{I_{\hat P}}} \\
  &\leq 2\abs{J}\sup_{P'\geq\tilde{P}}\frac{\abs{I_{P'}\cap E_{P'}}}{\abs{I_{P'}}}
  \leq 2\abs{J}\operatorname{density}(\mathbf{T}).\qedhere
\end{split}
\end{equation*}
\end{proof}

Next, divide $\mathbf{T}$ into the down- and up-trees
\begin{equation*}
  \mathbf{T}_d:=\{P\in\mathbf{T}:P\leq_d T\},\qquad \mathbf{T}_u:=\mathbf{T}\setminus\mathbf{T}_d,
\end{equation*}
and write
\begin{equation*}
  F_{jJ}:=\sum_{\substack{P\in\mathbf{T}_j\\ I_P\supsetneq J}}\epsilon_J\pair{f}{w_{P_d}}w_{P_d}1_{E_{P_u}},\qquad j\in\{d,u\}.
\end{equation*}

\begin{lemma}
\begin{equation*}
  \Norm{F_{dJ}}{L^1(J;X)}\leq \operatorname{size}(\mathbf{T})\abs{G_J}.
\end{equation*}
\end{lemma}

\begin{proof}
Suppose that $P,P'\in\mathbf{T}_d$ appear in the same sum $F_{dJ}$. Then $\omega_{P_d},\omega_{P_d'}\supseteq\omega_{T_d}$. If $\omega_{P_d}$ is the larger of the two, then $\omega_{P_d}\supsetneq\omega_{P_d'}$ and hence $\omega_{P_d}\supseteq\omega_{P'}$. Thus $\omega_{P_u}$ is disjoint from $\omega_{P'}$ and a fortiori from $\omega_{P_u'}$. And in particular the sets $E_{P_u}=E\cap\{N\in P_u\}$ and $E_{P_u'}$ are disjoint. Thus
\begin{equation*}
  \Norm{F_{dJ}}{\infty}
  =\sup_{\substack{P\in\mathbf{T}_d\\ I_P\supsetneq J}}\Norm{\pair{f}{w_{P_d}}w_{P_d}1_{E_{P_u}}}{\infty} 
  \leq\sup_{\substack{P\in\mathbf{T}_d\\ I_P\supsetneq J}}\frac{\abs{\pair{f}{w_{P_d}}}}{\abs{I_P}^{1/2}}\leq\operatorname{size}(\mathbf{T}).
\end{equation*}
Since $1_JF_{dJ}$ is supported on $G_J$, the claim follows.
\end{proof}

\begin{lemma}
\begin{equation*}
	\Norm{F_{uJ}}{L^1(J;X)}\leq 2\abs{G_J}\inf_{x\in J} M\tilde{f} (x),\qquad\tilde{f}:=\sum_{P\in\mathbf{T}_u}\epsilon_P\pair{f}{w_{P_d}}w_{P_d}.
\end{equation*}
\end{lemma}

\begin{proof}
Consider a fixed $x\in J$ with $F_{uJ}(x)\neq 0$. For the bitiles $P\in\mathbf{T}_u$, the sets $\omega_{P_u}$ are nested, and hence so are the sets $E_{P_u}$. The condition that $1_{E_{P_u}}(x)\neq 0$ is hence satisfied by all $P\in\mathbf{T}_u$ with $\omega_P$ large enough, hence $I_P$ not too large, say $I_P\subseteq I_x$. Thus
\begin{equation*}
\begin{split}
  F_{uJ}(x)&=\sum_{\substack{P\in\mathbf{T}_u\\ J\subsetneq I_P\subseteq I_x}}\epsilon_P\pair{f}{w_{P_d}}w_{P_d}(x) \\
   &=\sum_{\substack{P\in\mathbf{T}_u\\ J\subsetneq I_P\subseteq I_x}}\epsilon_P\epsilon_{PT}\pair{f}{w_{P_d}}w_{T_u}^\infty(x) h_{I_P}(x) \\
   &=w_{T_u}^\infty(x)(\mathbb{E}_J-\mathbb{E}_{I_x})\Big(
     \sum_{P\in\mathbf{T}_u}\epsilon_P\epsilon_{PT}\pair{f}{w_{P_d}}h_{I_P}\Big)(x) \\
   &=w_{T_u}^\infty(x)(\mathbb{E}_J-\mathbb{E}_{I_x})\Big(w_{T_u}^\infty
     \sum_{P\in\mathbf{T}_u}\epsilon_P\pair{f}{w_{P_d}}w_{P_d}\Big)(x).
\end{split}
\end{equation*}
By the unimodularity of $w_{T_u}^\infty$, from here we deduce that
\begin{equation*}
   \abs{F_{uJ}(x)} \leq 2\sup_{I\supseteq J}\fint_I\Babs{\sum_{P\in\mathbf{T}_u}\epsilon_P\pair{f}{w_{P_d}}w_{P_d}(y)}\ud y
  \leq 2\inf_J M\Big(\sum_{P\in\mathbf{T}_u}\epsilon_P\pair{f}{w_{P_d}}w_{P_d}\Big),
\end{equation*}
and the claim follows by using again that $\supp 1_J F_{uJ}\subseteq G_J$.
\end{proof}

We substitute these estimates to \eqref{eq:splitDomain}:
\begin{equation*}
\begin{split}
  \sum_{P\in\mathbf{T}} &\abs{\pair{f}{w_{P_d}}\pair{w_{P_d}}{g 1_{E_{P_u}}}}
   \leq\sum_{J\in\mathscr{J}}\Norm{F_{dJ}+F_{uJ}}{L^1(J;X)} \\
   &\leq\sum_{J\in\mathscr{J}}\abs{G_J}(\operatorname{size}(\mathbf{T})+2\inf_J M\tilde{f}) \\
   &\leq\sum_{J\in\mathscr{J}}2\operatorname{density}(\mathbf{T})\abs{J}\Big(\operatorname{size}(\mathbf{T})+2\inf_J M\tilde{f}\Big) \\
   &\leq 2\operatorname{density}(\mathbf{T})\operatorname{size}(\mathbf{T})\abs{I_T}
      +4\operatorname{density}(\mathbf{T})\int_{I_T}\ M\tilde{f}(x)\ud x.
\end{split}
\end{equation*}
The proof of Proposition~\ref{prop:treeLemma} is completed by
\begin{equation*}
\begin{split}
   \int_{I_T}\ M\tilde{f}(x)\ud x
   &\leq\abs{I_T}^{1/q'}\Norm{M\tilde{f}}{q}
   \leq C\abs{I_T}^{1/q'}\Norm{\tilde{f}}{q} \\
   &=C\abs{I_T}^{1/q'}\BNorm{\sum_{P\in\mathbf{T}_u}\epsilon_P\pair{f}{w_{P_d}}w_{P_d}}{L^q(\R_+;X)} \\
   &\overset{(*)}{\leq} C\abs{I_T}^{1/q'}\BNorm{\sum_{P\in\mathbf{T}_u}\pair{f}{w_{P_d}}w_{P_d}}{L^q(\R_+;X)} \\
   &\leq C\abs{I_T}^{1/q'}\abs{I_T}^{1/q}\operatorname{size}(\mathbf{T}),
\end{split}
\end{equation*}
where $(*)$ was an application of the UMD property, observing that
\begin{equation*}
   w_{T_u}^\infty\sum_{P\in\mathbf{T}_u}\epsilon_P\pair{f}{w_{P_d}}w_{P_d}
   =\sum_{P\in\mathbf{T}_u}\epsilon_P\epsilon_{PT}\pair{f}{w_{P_d}}h_{I_P}
\end{equation*}
is a martingale transform of the similar expression with all $\epsilon_P\equiv 1$.

\section{The density lemma}

\begin{proposition}
Every finite set $\mathbf{P}$ of bitiles has a disjoint decomposition
\begin{equation*}
  \mathbf{P}=\mathbf{P}_{\operatorname{sparse}}\cup\bigcup_j \mathbf{T}_j,
\end{equation*}
where each $\mathbf{T}_j$ is a tree, and
\begin{equation*}
  \operatorname{density}(\mathbf{P}_{\operatorname{sparse}})\leq 2^{-q}\operatorname{density}(\mathbf{P}),\qquad
  \sum_j\abs{I_{\mathbf{T}_j}}\leq 2^q\operatorname{density}(\mathbf{P})^{-1}\abs{E}.
\end{equation*}
\end{proposition}

\begin{proof}
Necessarily, we need to set
\begin{equation*}
  \mathbf{P}_{\operatorname{sparse}}
  :=\Big\{P\in\mathbf{P}: \sup_{P'\geq P}\frac{\abs{I_{P'}\cap E_{P'}}}{\abs{I_{P'}}} \leq 2^{-q}\operatorname{density}(\mathbf{P})\Big\}.
\end{equation*}
For every $P\in\mathbf{P}\setminus\mathbf{P}_{\operatorname{sparse}}$, we pick some bitile $P'$ such that
\begin{equation*}
  \frac{\abs{I_{P'}\cap E_{P'}}}{\abs{I_{P'}}} >2^{-q}\operatorname{density}(\mathbf{P}).
\end{equation*}
Let $T_j$ be the maximal bitiles (with respect to their partial order $\leq$) among these chosen $P'$, and let
\begin{equation*}
  \mathbf{T}_j:=\{P\in\mathbf{P}:P\leq T_j\}
\end{equation*}
be the tree in $\mathbf{P}$ with top $T_j$. Then
\begin{equation*}
  \mathbf{P}\setminus\mathbf{P}_{\operatorname{sparse}}=\bigcup_j\mathbf{T}_j.
\end{equation*}

Observe that the sets $I_{T_j}\cap E_{T_j}=I_{T_j}\cap E\cap\{N\in\omega_{T_j}\}$, which are all contained in $E$, are pairwise disjoint. Indeed, if two such sets intersected, then so would the corresponding bitiles $T_j=I_{T_j}\times\omega_{T_j}$, and then one of them could not be maximal. Thus we have
\begin{equation*}
  \sum_j\abs{I_{T_j}}
  \leq 2^q\operatorname{density}(\mathbf{P})^{-1}\sum_j\abs{I_{T_j}\cap E_{T_j}}\leq 2^q\operatorname{density}(\mathbf{P})^{-1}\abs{E}.\qedhere
\end{equation*}
\end{proof}

\section{The size lemma}

\begin{proposition}
Let $X$ be a UMD space with tile-type $q$. Then every finite set $\mathbf{P}$ of bitiles has a disjoint decomposition
\begin{equation*}
  \mathbf{P}=\mathbf{P}_{\operatorname{small}}\cup\bigcup_j \mathbf{T}_j,
\end{equation*}
where each $\mathbf{T}_j$ is a tree, and
\begin{equation*}
  \operatorname{size}(\mathbf{P}_{\operatorname{small}})\leq\tfrac12\operatorname{size}(\mathbf{P}),\qquad
  \sum_j\abs{I_{\mathbf{T}_j}}\leq C\operatorname{size}(\mathbf{P})^{-q}\Norm{f}{L^q(\R_+;X)}^q.
\end{equation*}
\end{proposition}

\begin{proof}
For every tree $\mathbf{T}$, let
\begin{equation*}
  \Delta(\mathbf{T})^q:=\frac{1}{\abs{I_{T}}}\int\Babs{\sum_{P\in\mathbf{T}_u}\pair{f}{w_{P_d}}w_{P_d}(x)}^q\ud x,
\end{equation*}
where $T$ is the top of $\mathbf{T}$, and $\mathbf{T}_u:=\{P\in\mathbf{T}:P\leq_u T\}$ is the up-tree supported by the same top.

Let $\sigma:=\operatorname{size}(\mathbf{P})$. We extract the trees $\mathbf{T}_j$ recursively as follows: Consider all maximal trees $\mathbf{T}\subseteq\mathbf{P}$ among the ones with $\Delta(\mathbf{T})>\frac 12 \sigma   $. Among them, let $\mathbf{T}_1$ be one whose top frequency interval $\omega_{\mathbf{T}}$ has the minimal center $c(\omega_{\mathbf{T}})$. Replace $\mathbf{P}$ by $\mathbf{P}\setminus\mathbf{T}_j$, and iterate. When no trees can be chosen anymore, the remaining collection $\mathbf{P}_{\operatorname{small}}$ satisfies $\operatorname{size}(\mathbf{P}_{\operatorname{small}})\leq\tfrac12\sigma$ by definition.

The sum over the top intervals is immediately estimated by
\begin{equation*}
  \sum_j\abs{I_{T_j}}\leq\frac{2^q}{\sigma^q}\sum_j\BNorm{\sum_{P\in\mathbf{T}_{j,u}}\pair{f}{w_{P_d}}w_{P_d}}{L^q(\R_+;X)}^q.
\end{equation*}
The sum on the right is bounded by $C\Norm{f}{L^q(\R_+;X)}^q$ as a direct application of the tile-type $q$ inequality, as soon as we verify the required disjointness condition that
\begin{equation}\label{eq:disjForTreeType}
  P_j\in\mathbf{T}_{j,u},\ P_i\in\mathbf{T}_{i,u},\ i\neq j\qquad\Longrightarrow\qquad P_{j,d}\cap P_{i,d}=\varnothing.
\end{equation}

Suppose to the contrary that for instance $P_{j,d}\leq P_{i,d}$,  and hence $\omega_{P_{i,d}}\subseteq\omega_{P_{j,d}}$. Since $P_i\neq P_j$, in fact  $\omega_{P_i}\subseteq\omega_{P_{j,d}}$. Thus, we have
\begin{equation*}
  \omega_{T_i}\subseteq\omega_{P_i}\subseteq\omega_{P_{j,d}},\qquad   \omega_{T_{j,u}}\subseteq\omega_{P_{j,u}},
\end{equation*}
and hence
\begin{equation*}
  c(\omega_{T_j})=\inf\omega_{T_{j,u}}\geq\inf\omega_{P_{j,u}}=\sup\omega_{P_{j,d}}>c(\omega_{T_i}).
\end{equation*}
This means that the tree $\mathbf{T}_i$ was chosen first, thus $i<j$. But $P_{j,d}\leq P_{i,d}$ implies $P_j\leq P_i\leq T_i$, so that $P_j$ should have been taken to $\mathbf{T}_i$ by maximality. This gives a contradiction, proving the claim \eqref{eq:disjForTreeType}, and hence the proposition.
\end{proof}

By using the density and size lemmas consecutively, it is easy to obtain the following:

\begin{lemma}
Suppose that
\begin{equation*}
  \operatorname{density}(\mathbf{P}_n)\leq 2^{nq}\abs{E},\qquad \operatorname{size}(\mathbf{P}_n)\leq 2^n\Norm{f}{q}.
\end{equation*}
Then
\begin{equation*}
  \mathbf{P}_n=\mathbf{P}_{n-1}\cup\bigcup_j\mathbf{T}_{n,j},\qquad
  \sum_j\abs{I_{\mathbf{T}_{j,n}}}\leq C2^{-nq},
\end{equation*}
where $\mathbf{P}_{n-1}$ satisfies estimates similar to $\mathbf{P}_n$ with $n-1$ in place of $n$.
\end{lemma}


If $\mathbf{P}$ is any finite collection of bitiles, it satisfies such estimates for some large $n$. By iteration, we obtain the decomposition
\begin{gather*}
  \mathbf{P}=\bigcup_{n\in\Z}\bigcup_j\mathbf{T}_{n,j},\qquad
  \\
  \operatorname{density}(\mathbf{T}_{n,j})\leq 2^{nq}\abs{E},\qquad\operatorname{size}(\mathbf{T}_{n,j})\leq 2^n\Norm{f}{q}, \qquad
  \sum_j\abs{I_{\mathbf{T}_{n,j}}}\leq C2^{-nq}.
  \end{gather*}
Note that there is also the trivial bound $\operatorname{density}(\mathbf{P})\leq 1$ for any collection.
And then
\begin{equation*}
\begin{split}
  \sum_{P\in\mathbf{P}}\abs{\pair{f}{w_{P_d}}\pair{w_{P_d}}{g1_{E_{P_u}}}}
  &\leq\sum_{n\in\Z}\sum_j\sum_{P\in\mathbf{T}_{n.j}}\ldots \\
  &\lesssim\sum_{n\in\Z}\sum_j\min\{1,2^{nq}\abs{E}\}\cdot 2^n\Norm{f}{q}\cdot\abs{I_{\mathbf{T}_{n,j}}} \\
  &\lesssim\sum_{n\in\Z}\min\{1,2^{nq}\abs{E}\}\cdot 2^n\Norm{f}{q}\cdot 2^{-nq} \lesssim\abs{E}^{1/q'}\Norm{f}{q}.
\end{split}
\end{equation*}

This shows that $\Norm{S^*f}{L^{q,\infty}}\lesssim\Norm{f}{L^q(0,1;Y)}$, proving the pointwise convergence $S_N f(x)\to f(x)$ for all $f\in L^q(0,1;Y)$. Note that $L^q$, where $q$ is the tile-type of $Y$, takes the classical role of $L^2$ as the space where estimates are easier than in general $L^p$ spaces.

\section{General $p>1$}
In this section we write $ C= S _{N (x)}$ for the Carleson operator.
In order to obtain the estimate
\begin{equation*}
  \Norm{Cf}{L^p(\R_+;X)}\lesssim\Norm{f}{L^p(\R_+;X)}
\end{equation*}
for all $p\in(1,\infty)$, we need to somewhat refine the previous considerations. First, we make the standard reduction: by interpolation, it suffices to prove the bound
\begin{equation*}
  \Norm{Cf}{L^{p,\infty}(\R_+;X)}\lesssim\Norm{f}{L^{p,1}(\R_+;X)}
\end{equation*}
for all $p\in(1,\infty)$, which by duality and a well-known description of the Lorentz space $L^{p,1}$ is equivalent to
\begin{equation*}
  \abs{\pair{Cf}{g}}\lesssim\abs{F}^{1/p}\abs{E}^{1/p'}
\end{equation*}
for all $f\in L^\infty(F;X)$, $g\in  L^\infty(E;X^*)$ bounded by one, and all bounded measurable sets $E$ and $F$. Yet another reduction is the following: It suffices that for every $E$ and $F$, we can find a \emph{major subset} $\tilde{E}\subseteq E$ with $\abs{\tilde{E}}\geq\tfrac12\abs{E}$ so that the previous estimate holds for all  $f\in L^\infty(F;X)$, $g\in  L^\infty(\tilde{E};X^*)$.

\begin{lemma}\label{lem:E<F}
Let $\abs{E}\leq\abs{F}$. Then
\begin{equation*}
  \abs{\pair{Cf}{g}}\lesssim\abs{E}\Big(1+\log\frac{\abs{F}}{\abs{E}}\Big)\lesssim\abs{E}^{1/p}\abs{F}^{1/p'}
\end{equation*}
for all $f\in L^\infty(F;X)$ and $g\in L^\infty(E;X^*)$ bounded by one, and all $p\in(1,\infty)$.
\end{lemma}

\begin{proof}
We observe an additional upper bound for every up-tree $\mathbf{T}$:
\begin{equation*}
\begin{split}
  \int\Babs{\sum_{P\in\mathbf{T}}\pair{f}{w_{P_d}}w_{P_d}(x)}^q\ud x
  &=\int\Babs{\sum_{P\in\mathbf{T}}\pair{fw_{T_u}^\infty}{h_{I_P}}h_{I_P}(x)}^q\ud x \\
  &\lesssim\Norm{f 1_{I_T}}{L^q(\R_+;X)}^q\leq\abs{I_T},
\end{split}
\end{equation*}
and hence $\operatorname{size}(\mathbf{P})\leq 1$. Thus
\begin{equation*}
\begin{split}
  \abs{\pair{Cf}{g}}
  &\lesssim \sum_{n\in\Z}\min\{1,2^{nq}\abs{E}\}\min\{1,2^n\abs{F}^{1/q}\}2^{-nq} \\
  &\leq\sum_{n:2^n\leq\abs{F}^{-1/q}}2^n\abs{E}\abs{F}^{1/q}
    +\sum_{n:\abs{F}^{-1/q}<2^n<\abs{E}^{-1/q}}\abs{E}
    +\sum_{n:\abs{E}^{-1/q}\leq 2^n}2^{-nq} \\
  &\lesssim\abs{E}\Big(1+\log\frac{\abs{F}}{\abs{E}}\Big).\qedhere
\end{split}
\end{equation*}
\end{proof}

The case $\abs{E}>\abs{F}$ is the more involved one. We need the following preparation:

\begin{lemma}\label{lem:sqFn}
Let $\mathscr{I}\subseteq\{I\in\mathscr{D}:\inf_I Mf\leq\lambda\}$ be a finite collection of dyadic intervals. Then
\begin{equation*}
  \BNorm{\sum_{\substack{I\in\mathscr{I}\\ I\subseteq K}}\pair{f}{h_I}h_I}{L^p(\R_+;X)}\lesssim \lambda\abs{K}^{1/p}.
\end{equation*}
\end{lemma}

\begin{proof}
Let
\begin{equation*}
  \tilde{f}:=\sum_{I\in\mathscr{I}}\pair{f}{h_I}h_I.
\end{equation*}
Then, denoting by $\mathscr{I}^*(K)$ the maximal elements $I\in\mathscr{I}$ with $I\subseteq K$,
\begin{equation*}
 1_K(\tilde{f}-\ave{\tilde{f}}_K)= \sum_{\substack{I\in\mathscr{I}\\I\subseteq K}}\pair{f}{h_I}h_I
  = \sum_{J\in\mathscr{I}^*(K)}\sum_{\substack{I\in\mathscr{I}\\ I\subseteq J} }\pair{f}{h_I}h_I,
\end{equation*}
which is a martingale transform of $1_{\bigcup\mathscr{I}^*(K)}f$. By the UMD property, these transforms are bounded from $L^1(\R_+;X)$ to $L^{1,\infty}(\R_+;X)$, and hence
\begin{equation*}
\begin{split}
  \Norm{1_K(\tilde{f}-\ave{\tilde{f}}_K)}{L^{1,\infty}(\R_+;X)}
  &\lesssim \Norm{1_{\bigcup\mathscr{I}^*(K)} f}{L^1(\R_+;X)} 
  \leq\sum_{J\in\mathscr{I}^*(K)}\Norm{1_J f}{L^1(\R_+;X)} \\
  &\leq\sum_{J\in\mathscr{I}^*(K)}\abs{J}\inf_J Mf
  \leq\sum_{J\in\mathscr{I}^*(K)}\abs{J}\lambda 	 \leq\lambda\abs{K}.
\end{split}
\end{equation*}
By the John--Str\"omberg inequality, we have $\Norm{\tilde{f}}{\BMO(\R_+;X)}\lesssim\lambda$, and then by the John--Nirenberg inequality that
\begin{equation*}
  \Norm{1_K(\tilde{f}-\ave{\tilde{f}}_K)}{L^p(\R_+;X)}\lesssim\lambda\abs{K}^{1/p}.\qedhere
\end{equation*}
\end{proof}

\begin{lemma}\label{lem:E>F}
Let $\abs{E}>\abs{F}$. Then there exists $\tilde{E}\subseteq E$ with $\abs{E}\leq 2\abs{\tilde{E}}$ such that
\begin{equation*}
  \abs{\pair{Cf}{g}}\lesssim \abs{F}\Big(1+\log\frac{\abs{E}}{\abs{F}}\Big). 
\end{equation*}
for all $f\in L^\infty(F;X)$ and $g\in L^\infty(\tilde{E};X^*)$ bounded by one.\end{lemma}

\begin{proof}
Let $G:=\{M(1_F)>2\abs{F}/\abs{E}\}$. Then $\abs{G}\leq\tfrac12\abs{E}$, and hence $\tilde{E}:=E\setminus G$ satisfies $\abs{\tilde{E}}\geq\tfrac12\abs{E}$. For $f$ and $g$ as in the assertion, we write
\begin{equation*}
  \sum_{P\in\mathbf{P}}\abs{\pair{f}{w_{P_d}}\pair{w_{P_d}}{g1_{E_{P_u}}}}
  =\sum_{\substack{P\in\mathbf{P}\\ I_P\not\subseteq G}}+ \sum_{\substack{P\in\mathbf{P}\\ I_P\subseteq G}},
\end{equation*}
and observe that the second sum vanishes. Indeed, $w_{P_d}$ is supported on $I_P\subseteq G$, and $g$ on $\tilde{E}\subseteq G^c$. For the first sum, we observe an additional upper bound for the size of any subset $\mathbf{P}'\subseteq\{P\in\mathbf{P}:I_P\not\subseteq G\}$: Let $\mathbf{T}\subseteq\mathbf{P}'$ be any up-tree with top $T$. Then for any $P\in\mathbf{T}$, we have
\begin{equation*}
  \inf_{I_P} M(f w_{T_u}^\infty)\leq\inf_{I_P} M(1_F)\leq 2\frac{\abs{F}}{\abs{E}},
\end{equation*}
since $I_P\not\subseteq G$. Hence, by Lemma~\ref{lem:sqFn},
\begin{equation*}
  \int\Babs{\sum_{P\in\mathbf{T}}\pair{f}{w_{P_d}}w_{P_d}(x)}^q\ud x
  =\int\Babs{\sum_{P\in\mathbf{T}}\pair{fw_{T_u}^\infty}{h_{I_P}}h_{I_P}(x)}^q\ud x
    \lesssim\Big(\frac{\abs{E}}{\abs{F}}\Big)^q\abs{I_T},
\end{equation*}
so that
\begin{equation*}
  \operatorname{size}(\mathbf{P}')\lesssim\frac{\abs{F}}{\abs{E}}.
\end{equation*}
Thus
\begin{equation*}
\begin{split}
  \abs{\pair{Cf}{g}} &\lesssim\sum_n\min\{1,2^{nq}\abs{E}\}\min\{\abs{F}/\abs{E},2^n\abs{F}^{1/q}\}\cdot 2^{-nq} \\
  &\leq\sum_{n:2^n\leq \abs{F}^{1/q'}/\abs{E}}\abs{E}\cdot 2^n\abs{F}^{1/q}
    +\sum_{n: \abs{F}^{1/q'}/\abs{E}<2^n<\abs{E}^{-1/q}}\abs{F} \\
   &\qquad +\sum_{n:\abs{E}^{-1/q}\leq 2^n}\abs{F}/\abs{E}\cdot 2^{-nq} \\
   &\lesssim\abs{F}\Big(1+\log\frac{\abs{E}}{\abs{F}}\Big).\qedhere
\end{split}
\end{equation*}
\end{proof}

Lemmas~\ref{lem:E<F} and \ref{lem:E>F} prove the reduced restricted weak-type estimate explained in the beginning of the section, and thereby complete the proof of our main Theorem~\ref{thm:main}.

\bibliographystyle{plain}

\end{document}